\documentclass[11pt]{article}
\pdfoutput=1
\usepackage{enumerate}
\usepackage{graphicx,fancyhdr,amsmath,amssymb,subfig,url}
\usepackage{algorithm}
\usepackage{todonotes}
\usepackage{algpseudocode}
\usepackage{booktabs}
\usepackage{float}
\restylefloat{table}
\usepackage{multirow}
\usepackage[normalem]{ulem}
\usepackage{graphicx}				
\usepackage{amssymb}
\usepackage{color}
\usepackage{xspace}
\usepackage[numbers,square]{natbib}
\usepackage{array}
\usepackage[top=1in, right=1in, left=1in, bottom=1in]{geometry}
\usepackage{amsthm}
\usepackage[hidelinks]{hyperref}

\def\colAwidth{3.5cm}
\def\colBwidth{3.5cm}
\def\colCwidth{1.5cm}
\def\colCmodwidth{1.5cm}

\usepackage{xifthen}

\newcommand{\ie}{\emph{i.e.}\ }

\def\AGD{Accelerated-gradient-descent}
\def\ACAGD{Almost-convex-AGD}
\def\ACCNC{Accelerated-non-convex-method}
\def\NCD{Negative-curvature-descent}

\newcommand{\callAGD}[1]{\hyperref[alg:AGD]{\Call{\AGD}{#1}}}
\newcommand{\callACAGD}[1]{\hyperref[alg:ACAGD]{\Call{\ACAGD}{#1}}}
\newcommand{\callACCNC}[1]{\hyperref[alg:ACCNC]{\Call{\ACCNC}{#1}}}
\newcommand{\callNCD}[1]{\hyperref[alg:NCD]{\Call{\NCD}{#1}}}

\def\SmGrad{L_1}
\def\SmHess{L_2}
\def\StrConv{\sigma_1}

\def\Dim{d}
\def\Reals{\mathbb{R}}
\newcommand{\func}{\ensuremath{f : \Reals^{\Dim} \rightarrow \Reals}\xspace}

\newcommand{\T}{\mathcal{T}}
\renewcommand{\T}{\mathsf{T}_{\rm grad}}
\newcommand{\TGrad}{\T}
\renewcommand{\TGrad}{\mathsf{T}_{\rm grad}}
\newcommand{\THess}{\mathcal{T}_H}
\renewcommand{\THess}{\mathsf{T}_{\rm hess}}

\def\rad{\frac{\alpha}{\SmHess }}
\def\radv{\alpha/\SmHess}
\def\DeltaF{\Delta_{f}}

\newcommand{\ceil}[1]{\lceil #1 \rceil}

\input{macros.sty}

\theoremstyle{plain}

\newtheorem{theorem}{Theorem}
\newtheorem{lemma}[theorem]{Lemma}
\newtheorem{corollary}[theorem]{Corollary}
\newtheorem{definition}{Definition}
\newtheorem{assumption}{Assumption}

\newcounter{remark}

\newenvironment{remark}[1][]{
	\refstepcounter{remark}
	\ifthenelse{\isempty{#1}}{%
		\noindent \textbf{Remark \theremark:}\hspace*{.05em}
	}{%
		\noindent \textbf{Remark \theremark} ({#1})\textbf{:}\hspace*{.05em}
	}
}{%
	$\clubsuit$ \bigskip
}

\theoremstyle{remark}

\numberwithin{theorem}{section}

\renewenvironment{proof}{\noindent{\bf 
		Proof.}\hspace*{1em}}{\qed\medskip\\}

\title{Accelerated Methods for Non-Convex Optimization}

\author{Yair Carmon ~~~ John C.\ Duchi ~~~ Oliver Hinder ~~~ Aaron 
Sidford\\
\texttt{\{\href{mailto:yairc@stanford.edu}{yairc},\href{mailto:jduchi@stanford.edu}{jduchi},\href{mailto:ohinder@stanford.edu}{ohinder},\href{mailto:sidford@stanford.edu}{sidford}\}@stanford.edu}}

\date{}

\begin{document}
\maketitle

\begin{abstract}
  We present an accelerated gradient method for non-convex optimization
  problems with Lipschitz continuous first and second derivatives. The method
  requires time $O(\epsilon^{-7/4} \log(1/ \epsilon) )$ to find an
  $\epsilon$-stationary point, meaning a point $x$ such that
  $\|\nabla f(x)\| \le \epsilon$. The method improves upon the
  $O(\epsilon^{-2} )$ complexity of gradient descent and provides the
  additional second-order guarantee that
  $\nabla^2 f(x) \succeq -O(\epsilon^{1/2})I$ for the computed
  $x$. Furthermore, our method is Hessian-free, \ie it only requires gradient
  computations, and is therefore suitable for large scale applications.
\end{abstract}

\section{Introduction}

In this paper, we consider the optimization problem
\begin{flalign}
  \label{eqn:main-objective}
  \minimize_{x\in\Reals^\Dim} f(x),
\end{flalign}
where \func has $\SmGrad$-Lipschitz continuous gradient and
$\SmHess$-Lipschitz continuous Hessian, but may be non-convex. Without further
assumptions, finding a global minimum of this problem is
computationally intractable: finding an $\epsilon$-suboptimal point for a
$k$-times continuously differentiable function $\func$ requires at least
$\Omega( (1/\epsilon)^{\Dim / k} )$ evaluations of the function and first
$k$-derivatives, ignoring problem-dependent constants~\cite[\S
1.6]{NemirovskiYu83}.  Consequently, we aim for a weaker guarantee, looking
for locally ``optimal'' points for problem~\eqref{eqn:main-objective}; in
particular, we seek \emph{stationary points}, that is points $x$ with sufficiently small
gradient:
\begin{flalign}
  \label{eq:gradient-norm-squared}
  \norm{\nabla f(x)} \le \epsilon.
\end{flalign}

The simplest method for obtaining a guarantee of the
form~\eqref{eq:gradient-norm-squared} is gradient descent (GD). It is
well-known~\citep{NesterovGradSmall2012} that if GD begins from a point $x_1$, then
for any $\DeltaF \ge f(x_1) - \inf_x f(x)$ and $\epsilon > 0$,
it finds a point
satisfying the bound~\eqref{eq:gradient-norm-squared} in $O(\DeltaF \SmGrad \epsilon^{-2})$ iterations. If one additionally
assumes that the function $f$ is convex, substantially more is possible: GD
then requires only $O( R \SmGrad \epsilon^{-1})$ iterations, where $R$ is an
upper bound on the distance between $x_1$ and the set of
minimizers of $f$.  Moreover, in the same note~\cite{NesterovGradSmall2012},
Nesterov also shows that
acceleration and regularization techniques can reduce the iteration complexity
to $\wt{O}( \sqrt{R \SmGrad} \epsilon^{-1/2})$.\footnote{The notation $\wt{O}$
  hides logarithmic factors. See
  Definition~\ref{big-O}.}

In the non-convex setting, it is possible to achieve better rates of
convergence to stationary points assuming access to more than gradients,
\emph{e.g.}\ the full Hessian.  \citet{nesterov2006cubic} explore such
possibilities with their work on the cubic-regularized Newton method, which
they show computes an $\epsilon$-stationary point in
$O(\DeltaF \SmHess^{0.5} \epsilon^{-3/2})$ iterations (\ie gradient and
Hessian calculations). However, with a naive implementation, each such
iteration requires explicit calculation of the Hessian $\nabla^2 f(x)$ and the
solution of multiple linear systems, with complexity
$\wt{O}(\Dim^3)$.\footnote{Technically, $\wt{O}(d^\omega)$ where
  $\omega < 2.373$ is the matrix multiplication constant
  \cite{williams2012multiplying}.} More recently, \citet{birgin2015worst}
extend cubic regularization to $p$th order regularization, showing that
iteration complexities of $O(\epsilon^{-(p+1)/p})$ are possible given
evaluations of the first $p$ derivatives of $f$. That is, there exist
algorithms for which $\epsilon^{-(p + 1)/p}$ calculations of the first $p$
derivatives of $f$ are sufficient to achieve the
guarantee~\eqref{eq:gradient-norm-squared}; naturally, these bounds ignore the
computational cost of each iteration. More efficient rates are also known for
various structured problems, such as finding KKT points for indefinite
quadratic optimization problems~\citep{ye1998complexity} or local minima of
$\ell_p$ ``norms,'' $p \in (0, 1)$, over linear
constraints~\citep{ge2011note}.

In this paper, we ask a natural question: using only gradient
information, is it possible to improve on the $\epsilon^{-2}$ iteration
complexity of gradient descent in terms of number of gradient calculations? We
answer the question in the affirmative, providing an algorithm that requires at 
most
\begin{equation*}
  \wt{O}\left( \DeltaF
    \SmGrad^{\half}\SmHess^{\frac{1}{4}}\epsilon^{-\frac{7}{4}} + \DeltaF^{1/2} \SmGrad^{1/2} \epsilon^{-1} 
    \right)
\end{equation*}
gradient and Hessian-vector product evaluations to find an $x$ such that 
\mbox{$\norm{\nabla f(x)} \le \epsilon$}.
For a summary of our 
results 
in relation
to other work, see Table~\ref{runtime-table}.

Another advantage of the cubic-regularized Newton method is that it provides a
second-order guarantee of the form $\nabla^2 f(x) \succeq -\sqrt{\epsilon}I$,
thus giving a rate of convergence to points with zero gradient and positive
semi-definite Hessian. Such second-order stationary points are finer
approximations of local minima compared to first-order stationary points (with
zero gradient). Our approach also provides this guarantee, and is therefore an
example of a \emph{first-order} method that converges to a \emph{second-order}
stationary point in time polynomial in the desired accuracy and with
logarithmic dependence on the problem dimension. A notable consequence of this
approach is that for strict saddle
functions~\cite{lee2016gradient,ge2015escaping}---with only non-degenerate 
stationary points---our approach converges linearly to
local minimizers. We discuss this result in detail in
Section~\ref{sec:strict-saddle}.

\begin{table}[H] \caption{Runtime comparisons for finding
		a first-order stationary point~\eqref{eq:gradient-norm-squared}}
	\label{runtime-table} \footnotesize	
	\begin{center}
		\begin{tabular}{ | >{\centering\arraybackslash}m{\colAwidth} | 
		>{\centering\arraybackslash}m{\colBwidth} | 
		>{\centering\arraybackslash}m{\colCwidth} | 
		>{\centering\arraybackslash}m{\colCmodwidth} | 
		>{\centering\arraybackslash}m{\colCwidth} | 
		>{\centering\arraybackslash}m{\colCwidth} | }
			\hline
			& \# iterations & Hessian free? & Gradient Lipschitz? & Hessian 
			Lipschitz? &  convex $f$? \\
			\hline
			Gradient descent (non-convex case)& $O\left(  \DeltaF   \SmGrad 
			\epsilon^{-2} \right)$ & Yes & Yes & No & No \\ 
			\hline
			Gradient descent (convex case)  \cite{NesterovGradSmall2012}& 
			$O\left( R\SmGrad \epsilon^{-1} \right)$ & Yes & Yes & No & Yes \\ 
			\hline
			Proximal accelerated gradient descent  \cite{NesterovGradSmall2012} 
			& $\wt{O}\left(  (R \SmGrad)^{\half} \epsilon^{-\half} \right)$ & Yes   & 
			Yes & No & Yes \\
			\hline
			Cubic-regularized Newton method  \cite{nesterov2006cubic}  & 
			$\wt{O}\left(   \DeltaF  \SmHess^{\half}  \epsilon^{-\frac{3}{2}} \right)$ & No 
			& Yes & Yes & No \\ 
			\hline
			\textbf{This paper (Theorem~\ref{thm:complexity-of-ACCNC})} & 
			$\boldsymbol{\wt{O}\left(  \DeltaF \SmGrad^{\half}  \SmHess^{\frac{1}{4}}
			\epsilon^{-\frac{7}{4}} \right)}$ & \textbf{Yes} & \textbf{Yes} & \textbf{Yes} & 
			\textbf{No} \\
			\hline
		\end{tabular}
	\end{center}
\end{table}

\subsection{Related work and background}

In the optimization and machine learning literature, there has been
substantial recent work on the convergence properties of optimization methods
for non-convex problems.  One line of work investigates the types of local
optima to which gradient-like methods converge, as well as convergence
rates. In this vein, under certain reasonable assumptions (related to
geometric properties of saddle points), \citet{ge2015escaping} show that
stochastic gradient descent (SGD) converges to second-order local optima
(stationary points with positive semidefinite Hessian), while
\citet{lee2016gradient} show that GD generically converges to second-order
local optima.  \citet{anandkumar2016efficient} extend these ideas, showing how
to find a point that approximately satisfies the third-order necessary
conditions for local optimality in polynomial time. While these papers used
second-order smoothness assumptions to ensure convergence to stronger local
minima than the simple stationary condition~\eqref{eq:gradient-norm-squared},
they do not improve rates of convergence to stationarity.
  
A second line of work focuses on improving the slow convergence rates of SGD
to stationary points (typically $O(\epsilon^{-4})$ stochastic gradient
evaluations are sufficient~\cite{GhadimiLa13}) under appropriate structural
conditions on $f$. One natural condition---common in the statistics and
machine learning literature---is that $f$ is the sum of $n$ smooth non-convex
functions. Indeed, the work of \citet{reddi2016stochastic} and
\citet{allen2016variance} achieves a rate of convergence $O(\epsilon^{-2})$ for
such problems without performing the $n$ gradient evaluations (one per
function) that standard gradient descent requires in each iteration. These analyses extend
variance-reduction techniques that apply to incremental
convex optimization problems~\cite{JohnsonZh13,DefazioBaLa14}. Nonetheless,
they do not improve on the $O(\epsilon^{-2})$ iteration complexity of
GD.

Additionally, a number of researchers apply accelerated gradient
methods~\citep{Nesterov04} to non-convex optimization problems, though we
know no
theoretical guarantees giving improved performance over standard
gradient descent methods. \citet{ghadimi2016accelerated} show how to modify
Nesterov's accelerated gradient descent method so that it enjoys the same
convergence guarantees as gradient descent on non-convex optimization
problems, while maintaining the accelerated (optimal) first-order convergence
rates for convex problems. \citet{NIPS2015_5728} develop an accelerated method
for non-convex optimization and show empirically that on (non-convex) sparse
logistic regression test problems their methods outperform other methods,
including gradient descent.

While the subproblem that appears in the cubic-regularized Newton method is 
expensive to solve exactly, it is possible to consider methods in which such 
subproblems are solved only approximately by a low complexity Hessian-free 
procedure.
A number of researchers investigate this approach, including
\citet{cartis2011adaptive} and
\citet{bianconcini2015use}. These works exhibit strong empirical results, but their 
analyses do not improve on the $O(\epsilon^{-2})$ evaluation complexity of 
gradient descent. Recently, \citet{hazan2016linear} and 
Ho-Nguyen and K{\i}l{\i}nc-Karzan~\cite{Ho-NguyenKi16} have
shown how to solve the related quadratic non-convex trust-region problem using
accelerated first-order methods; both these papers use accelerated eigenvector
computations as a primitive, similar to our approach. It is therefore natural to 
ask whether acceleration can give faster 
convergence to stationary points of general non-convex functions, a question
we answer in the affirmative.
 
Concurrently to and independently of this paper,\footnote{A preprint of the
  current paper appears on the \texttt{arXiv}~\cite{CarmonDuHiSi16}.}
\citet{agarwal2016finding} also answer this question affirmatively.  They
develop a method that uses fast approximate matrix inversion as a primitive to
solve cubic-regularized Newton-type steps~\cite{nesterov2006cubic}, and applying
additional acceleration techniques they show how to find stationary points of
general smooth non-convex objectives.  Though the technical approach is
somewhat different, their convergence rates to $\epsilon$-stationary
points are identical to ours. They also specialize their technique to problems
of the finite sum form $f = \frac{1}{n} \sum_{i = 1}^n f_i$, showing that additional
improvements in terms of $n$ are achievable.

\subsection{Our approach}

Our method is in the spirit of the techniques that underly accelerated
gradient descent (AGD). While \citeauthor{nesterov1983method}'s
\citeyear{nesterov1983method} development of acceleration schemes 
may appear mysterious at first, there are multiple
interpretations of AGD as the careful combination of different routines for
function minimization. The estimate sequence ideas of \citet{Nesterov04} and
proximal point proofs \citep{lin2015universal,shalev2014accelerated,
  frostig2015regularizing} show how to view accelerated gradient descent as a
trade-off between building function lower bounds and directly making function
progress. \citet{bubeck2015geometric} develop an AGD algorithm with a 
geometric interpretation based on shrinking spheres, while the work
of \citet{allen2014linear} shows that AGD may be viewed as a coupling between
mirror descent~\citep{NemirovskiYu83} and gradient descent; this perspective
highlights how to trade each method's advantages in different scenarios to
achieve faster---accelerated---running time.

We follow a similar template of leveraging two competing techniques for making
progress on computing a stationary point, but we deviate from standard analyses
involving acceleration in our coupling of the algorithms. The first technique
we apply is fairly well known. If the problem is locally non-convex, the
Hessian must have a negative eigenvalue. In this case, under the assumption
that the Hessian is Lipschitz continuous, moving in the direction of the corresponding
eigenvector \emph{must} make progress on the objective.
\citet{nesterov2006cubic} (and more broadly, the literature on cubic
regularization) use this implicitly, while other
researchers~\citep{ge2015escaping,anandkumar2016efficient} use this more
explicitly to escape from saddle points.

The second technique is the crux of our approach. While $\SmGrad$-Lipschitz
continuity of $\nabla f$ ensures that the smallest eigenvalue of the Hessian
is at least $-\SmGrad$, we show that any stronger bound---any deviation
from this ``worst possible'' negative curvature---allows us to improve upon
gradient descent. We show that if the smallest eigenvalue is at least
$-\gamma$, which we call $-\gamma$-strong convexity, we can apply proximal
point techniques~\cite{parikh2014proximal, NesterovGradSmall2012} and
accelerated gradient descent to a carefully constructed regularized problem to
obtain a faster running time. Our procedure proceeds by approximately
minimizing a sequence of specially constructed such functions. This procedure 
is of independent interest since it can be applied in a standalone manner whenever the
function is known to be globally almost convex.

By combining these procedures, we achieve our result. We run an accelerated
(single) eigenvector routine---also known as principle components analysis
(PCA)---to estimate the eigenvector corresponding to the smallest
eigenvalue of the Hessian. Depending on the estimated eigenvalue we either
move along the approximate eigenvector or apply accelerated gradient descent
to a regularized sub-problem, where we carefully construct the regularization
based on this smallest eigenvalue.  Trading between these two cases gives our
improved running time. We remark that an improvement over gradient descent is
obtainable even if we use a simpler (non-accelerated) method for estimating
eigenvectors, such as the power method. That said,
an accelerated gradient descent subroutine for the regularized sub-problems we
solve appears to be crucial to achieving faster convergence rates than
gradient descent.

The remainder of the paper is structured as follows.
Section~\ref{sec:definitions} introduces the notation and existing results on
which our approach is based. Section~\ref{sec:almost-convex} introduces our
method for accelerating gradient descent on ``almost convex'' functions, while
Section~\ref{sec:negative-curvature} presents and explains our ``negative
curvature descent'' subroutine. Section~\ref{sec:main-theorem} joins the two
building blocks to obtain our main result, while in Section~\ref{sec:strict-saddle},
we show how our results give linear convergence to local
minima for strict-saddle functions.

\section{Notation and standard results}\label{sec:definitions}

Here, we collect our (mostly standard) notation and a few basic
results.  Throughout this paper,
norms $\norm{\cdot}$ are the Euclidean norm; when applied to matrices
$\norm{\cdot}$ denotes the $\ell_2$-operator norm. All logarithms are
base-$e$. For a symmetric matrix $A$, we let $\lambda_{\min}(A)$ and
$\lambda_{\max}(A)$ denote its minimum and maximum eigenvalues,
respectively. We also use the following definitions.

\begin{definition}[Smoothness]
  \label{def:smoothness}
  A function \func is \emph{$\SmGrad$-smooth} if its gradient is
  $\SmGrad$-Lipschitz, that is,
  $\norm{\nabla f(x) - \nabla f(y)} \le \SmGrad \norm{x - y}$ for all $x, y$.
\end{definition}

\begin{definition}[Lipschitz Hessian]
  The Hessian of a twice differentiable function \func is
  \emph{$\SmHess$-Lipschitz continuous} if
  $\norm{\nabla^2 f(x) - \nabla^2 f(y)} \le \SmHess \| x - y \|$ for all
  $x, y$.
\end{definition}

\begin{definition}[Optimality gap]
  A function \func has \emph{optimality gap $\DeltaF > 0$ at point $x$} if
  $f(x)-\inf_{y\in \Reals^\Dim}f(y) \leq \DeltaF$.
\end{definition}
\noindent
We assume throughout without further mention
that $f$ is $\SmGrad$-smooth, has $\SmHess$-Lipschitz
continuous Hessian, and has optimality gap $\DeltaF < \infty$ at the initial search
point, generally denoted $z_1$.

The next definition is atypical, because we allow the strong convexity
parameter $\StrConv$ to be negative. Of course, if $\StrConv < 0$ the function
may be non-convex, but we can use $\StrConv$ to bound the extent to which the
function is non-convex, similar to the ideas of lower $C^2$-functions in
variational analysis~\cite{RockafellarWe98}.  As we show in
Lemma~\ref{lem:almost-convex} this ``almost convexity'' allows improvements in
runtime over gradient descent.
\begin{definition}[Generalized strong convexity and almost
  convexity]\label{def:strconvex}
  A function \func is \emph{$\StrConv$-strongly convex} if
  $\frac{\StrConv}{2} \| y-x \|^{2} \le f(y)-f(x) - \nabla f(x)^{T}(y-x)$ for
  some $\StrConv \in \R$. For $\gamma = \max\{-\StrConv, 0\}$, we call such
  functions \emph{$\gamma$-almost convex}.
\end{definition}

The next three results are standard but useful lemmas using the definitions
above.

\begin{lemma}[Nesterov~\cite{Nesterov04}, Theorem 2.1.5]
  \label{lem:grad_smoothness_prop} 
  Let \func be $\SmGrad$-smooth. Then for all $x,y \in \Reals^{\Dim}$
  \begin{equation*}
    | f(y)-f(x) - \nabla f(x)^{T} (y-x) | \le \frac{\SmGrad}{2} \| y-x \|^{2}
  \end{equation*}
\end{lemma}

\begin{lemma}[Nesterov and Polyak~\cite{nesterov2006cubic}, Lemma 1]
  \label{lem:hess_smoothness_prop}
  Let $f$ have $\SmHess$-Lipschitz Hessian. Then for all $x,y \in \Reals^{\Dim}$
  \begin{equation*}
    \| \nabla f(y) -  \nabla f(x) -  \nabla^2 f(x) (y-x) \| \le \frac{\SmHess}{2} \| y-x \|^{2}
  \end{equation*}
  and
  \begin{equation*}
    \left| f(y)- f(x) - \nabla f(x)^{T}(y-x) - \frac{1}{2} (y-x)^T \nabla^2 f(x) (y-x) \right| \le \frac{\SmHess}{6} \| y-x \|^{3}
  \end{equation*}
\end{lemma}

\begin{lemma}[Boyd and Vandenberghe~\cite{BoydVa04}, Eqs.~(9.9) and (9.14)]
  \label{lem:grad:bounds}
  Let $f$ be $\SmGrad$-smooth and
  $\mu$-strongly convex. Then for all $x$ the
  minimizer $x^*$ of $f$ satisfies
  \begin{equation*}
    2\mu [f(x) - f(x^*)]
    \leq 
    \norm{\nabla f(x)}^2
    \leq 2 \SmGrad [f(x) - f(x^*)].
  \end{equation*}
\end{lemma}

Lemma~\ref{lem:grad_smoothness_prop} guarantees any
$\SmGrad$-smooth function is $(-\SmGrad)$-strongly convex. A key idea in this
paper is that if a function is $(-\StrConv)$-strongly convex with
$\StrConv \ge 0$, standard convex proximal methods are still applicable,
provided the regularization is sufficiently large. The following trivial
lemma, stated for later reference, captures this idea.

\begin{lemma}\label{lem:strong-convexity}
  Suppose \func is $(-\StrConv)$-strongly convex, where $\StrConv \ge 0$.
  Then for any $x_0 \in \Reals^{\Dim}$ the function
  $g(x) = f(x) + \StrConv \| x - x_0 \|^2$ is $(\StrConv)$-strongly convex.
\end{lemma}

Throughout this paper, we use a fully
non-asymptotic big-O notation to be clear about the convergence rates of the
algorithms we analyze and to avoid confusion involving relative values of
problem-dependent constants (such as $d, \SmGrad, \SmHess$).

\begin{definition}[Big-O notation]
  \label{big-O}
  Let $\mc{S}$ be a set and let $M_1, M_2 : \mc{S} \to \R_+$.
  Then $M_1 = O(M_2)$ if there exists $C \in \R_+$ such that
  $M_1(s) \leq C\cdot M_2(s)$ for all $s \in
  \mc{S}$. 
  \end{definition}
\noindent
Throughout, we take $\mathcal{S} \subset [0,\infty)^6$ to be the set of tuples
$(\epsilon, \delta, \SmGrad, \SmHess, \DeltaF, \Dim)$; sometimes we 
require the tuples to meet certain assumptions that we specify. The notation
$\wt{O}(\cdot)$ hides logarithmic factors in problem
parameters: we say that $M_1 = \wt{O}(M_2)$ if
$M_1 = O(M_2 \log(1 + \SmGrad + \SmHess + \DeltaF + \Dim + 1/\delta + 
1/\epsilon))$.

Because we focus on gradient-based procedures, we measure the running time of
our algorithms in terms of gradient operations, each of which we assume takes
a (problem-dependent) amount of time $\T$. The following assumption
specifies this more precisely.
\begin{assumption}
  \label{assumption:grad-time}
  The following operations take $O(\T)$ time:
  \begin{enumerate}[1.]
  \item The evaluation $\nabla f(x)$ for a point $x \in \Reals^{\Dim}$.
  \item The evaluation of $\nabla^2 f(x) v$ for some vector
    $v \in \Reals^{\Dim}$ and point $x \in \Reals^{\Dim}$. (See
    Remark~\ref{remark:Hessian-products-and-gradients} for justification of
    this assumption.)
  \item Any arithmetic operation (addition, subtraction or multiplication) of
    two vectors of dimension at most $\Dim$.
  \end{enumerate}
\end{assumption}
\noindent
Based on Assumption~\ref{assumption:grad-time},
we call an algorithm \emph{Hessian free} if its basic operations take time at most $O(\T)$. 

\begin{remark}
  \label{remark:Hessian-products-and-gradients}
  By definition of the Hessian, we have
  that $\lim_{h \to 0} h^{-1} (\nabla f(x + hv) - \nabla f(x)) = \nabla^2 f(x)
  v$
  for any $v \in \R^{\Dim}$. Thus, a natural approximation to the product
  $\nabla^2 f(x) v$ is to set
  \begin{equation*}
    p = \frac{\nabla f(x + h v) - \nabla f(x)}{h}
  \end{equation*}
  for some small $h > 0$.  By Lemma~\ref{lem:hess_smoothness_prop}, we
  immediately have
  \begin{equation*}
    \norm{ p - \nabla^2 f(x) v } \le h \frac{\SmHess \| v \|^2}{2},
  \end{equation*}
  which allows sufficiently precise calculation by taking $h$ small.\footnote{
    We assume infinite precision arithmetic in this paper: see discussion in Section~\ref{sub:fast-eig}.
  }
		
  In a number of concrete cases, Hessians have structure that allows efficient
  computation of the product $v \mapsto \nabla^2 f(x) v$. For example, in
  neural networks, one may compute $\nabla^2 f(x) v$ using a
  back-propagation-like technique at the cost of at most two gradient
  evaluations~\cite{pearlmutter1994fast,schraudolph2002fast}.
\end{remark}

\subsection{Building block 1: fast gradient methods}

With the basic lemmas and definitions in place, we now recapitulate some of
the classical development of accelerated methods. First, the following
pseudo-code gives Nesterov's classical accelerated gradient descent method for
strongly convex functions~\cite{Nesterov04}.

\begin{algorithm}[H]
\begin{algorithmic}[1]\label{alg:AGD}
\Function{\AGD}{$f$, $y_{1}$, $\epsilon$, $\SmGrad$, $\StrConv$}
\State Set $\kappa = \SmGrad/\StrConv$ and $z_{1} = y_{1}$
\For{$j =  1, 2, \ldots$}
\If{$\| \nabla f(y_{j}) \| \le \epsilon$} \Return{$y_{j}$}
\EndIf
\begin{flalign*}
& y_{j+1} = z_{j} - \frac{1}{\SmGrad} \nabla f (z_j) \\
& z_{j+1} = \left(1 + \frac{\sqrt{\kappa} - 1}{\sqrt{\kappa} + 1} \right) y_{j+1} - \frac{\sqrt{\kappa} - 1}{\sqrt{\kappa} + 1} y_j
\end{flalign*}
\EndFor
\EndFunction
\end{algorithmic}
\end{algorithm}

The method \callAGD{} enjoys the following essentially standard guarantee when
initialized at any iterate $z_1$ satisfying
$f(z_1) - \inf_x f(x) \le \DeltaF$.

\begin{lemma}
  \label{lem:AGD}
  Let \func be $\StrConv > 0$-strongly convex and $\SmGrad$-smooth.
  Let $\epsilon > 0$ and let $z_j$ denote the
  $j$th iterate of
  \callAGD{$f$, $z_{1}$, $\epsilon$, $\SmGrad$, $\StrConv$}.
  If
  \begin{equation*}
    j \ge 1 + \sqrt{\frac{\SmGrad}{\StrConv}}
   \log \left( \frac{4 \SmGrad^2 \DeltaF }{ \StrConv \epsilon^2
        } \right) 
    ~~ \mbox{then} ~~
    \norm{\nabla f(z_j)} \le \epsilon.
  \end{equation*}
\end{lemma}

\begin{proof}
  Let $z^*$ be the minimizer of $f$. If
  $\epsilon^2 \ge 4\SmGrad^2 \DeltaF / \StrConv$, then
  \begin{equation*}
    \norm{\nabla f(z_{1})}^2 
    \stackrel{(i)}{\le} \SmGrad^2 \norm{z_1 - z^*}^2
    \stackrel{(ii)}{\le}
    \frac{2\SmGrad^2}{\StrConv}(f(z_1)-f(z^*))
    \stackrel{(iii)}{\le} \frac{4 \SmGrad^2 \DeltaF }{ \StrConv } \le \epsilon^2,
  \end{equation*}
  where inequality~$(i)$ follows from smoothness of $f$
  (Def.~\ref{def:smoothness}),
  inequality~$(ii)$ by the strong convexity of $f$
  (Lemma~\ref{lem:strong-convexity}),
  and inequality~$(iii)$ by the definition of $\DeltaF$. Thus
  the iteration ends at $j = 1$.

  For smaller $\epsilon$, we let $\kappa = \SmGrad / \StrConv \ge 1$ denote
  the condition number for the problem. Then \citet[Theorem 2.2.2]{Nesterov04}
  shows that for $k > 1$
  \begin{equation*}
    f(z_k) - f(z^*) \le 
    \SmGrad \left(1 - \sqrt{\StrConv / \SmGrad}\right)^{k-1} \norm{z_1 - z^*}^2
    \le 2 \kappa \exp(-(k-1) \kappa^{-\half}) \DeltaF.
  \end{equation*}
  Taking any $k \ge 1 + \sqrt{\kappa} \log \frac{4 \SmGrad \kappa \DeltaF}{\epsilon^2}$
  yields
  \begin{equation*}
    f(z_k) - f(z^*) \le \frac{\epsilon^2}{2 \SmGrad}.
  \end{equation*}
  Noting that $\norm{\nabla f(x)}^2 \le 2 \SmGrad (f(x) - f(z^*))$ by
  Lemma~\ref{lem:grad:bounds}, we obtain our result.
\end{proof}

\subsection{Building block 2: fast eigenvector computation}\label{sub:fast-eig}

The final building block we use is accelerated approximate leading
eigenvector computation. We consider two types of approximate eigenvectors. By
a \emph{relative $\varepsilon$-approximate leading eigenvector} of a positive
semidefinite (PSD) matrix $H$, we mean a vector $v$ such that $\norm{v} = 1$
and $v^T H v \ge (1 - \varepsilon) \lambda_{\max}(H)$; similarly, an
\emph{additive $\varepsilon$-approximate leading eigenvector} of $H$ satisfies
$\norm{v} = 1$ and $v^T H v \ge \lambda_{\max}(H) - \varepsilon$. A number of
methods compute such (approximate) leading eigenvectors, including the Lanczos
method~\cite{kuczynski1992estimating}. For concreteness, we state one lemma
here, where in the lemma we let $\THess$ denote the larger of the times
required to compute the matrix-vector product $Hv$ or to add two vectors.

\begin{lemma}[Accelerated top eigenvector computation]
  \label{lem:eigenvalue-computation-time}
  Let $H \in \Reals^{d \times d}$ be symmetric and PSD.
  There exists an algorithm that on input $\varepsilon, \delta \in (0, 1)$ runs
  in $O( \THess \log ( \Dim / \delta ) \varepsilon^{-1/2}) $ time and,
  with probability at least $1 - \delta$, returns a
  relative $\varepsilon$-approximate leading eigenvector $\what{v}$.
\end{lemma}

Notably, the Lanczos method~\cite[Theorem 3.2]{kuczynski1992estimating}
achieves this complexity guarantee. While
Lemma~\ref{lem:eigenvalue-computation-time} relies on infinite precision
arithmetic (the stability of the Lanczos method is an active area of research
\cite{orecchia2012approximating}), shift-and-invert
preconditioning~\cite{garber2016faster} also achieves the convergence
guarantee to within poly-logarithmic factors in bounded precision
arithmetic. This procedure reduces computing the top eigenvector of the matrix
$H$ to solving a sequence of linear systems, and using fast gradient descent
to solve the linear systems guarantees the running time in
Lemma~\ref{lem:eigenvalue-computation-time}. (See Section~8 of
\cite{garber2016faster} for the reduction and Theorem 4.1 of
\cite{zhuEvenFaster} for another statement of the result.)  For
simplicity---because we do not focus on such precision issues---we use
Lemma~\ref{lem:eigenvalue-computation-time} in the sequel.

For later use, we include a corollary of
Lemma~\ref{lem:eigenvalue-computation-time} in application to finding minimum
eigenvectors of the Hessian $\nabla^2 f(x)$ using matrix-vector multiplies.
Recalling that $f$ is $\SmGrad$-smooth, we know that the matrix
$M \defeq \SmGrad I - \nabla^2 f(x)$ is PSD, and its eigenvalues are
$\{\SmGrad I - \lambda_i\}_{i = 1}^\Dim \subset [0, 2 \SmGrad]$, where
$\lambda_i$ is the $i$th eigenvalue of $\nabla^2 f(x)$. The procedure
referenced in Lemma~\ref{lem:eigenvalue-computation-time} (Lanczos or another
accelerated method) applied to the matrix $M$ thus, with probability
at least $1 - \delta$, provides a vector
$\what{v}$ with $\norms{\what{v}} = 1$ such that
\begin{equation*}
  \SmGrad - \what{v}^T \nabla^2 f(x) \what{v}
  = \what{v}^T M \what{v}
  \ge (1 - \varepsilon) \lambda_{\max}(M) \ge (1 - \varepsilon) (\SmGrad -
  \lambda_{\min}(\nabla^2 f(x)))
\end{equation*}
in time $O(\TGrad \varepsilon^{-\half} \log \frac{d}{\delta})$. Rearranging this,
we have
\begin{equation*}
  \what{v}^T \nabla^2 f(x) \what{v}
  \le \varepsilon \SmGrad + (1 - \varepsilon) \lambda_{\min}(\nabla^2 f(x)),
\end{equation*}
and substituting $\epsilon / (2 \SmGrad)$ for $\varepsilon$ yields the following
summarizing corollary.

\begin{corollary}[Finding the negative curvature]
  \label{cor:smallest_evec_computation}
  In the setting of the previous paragraph, there exists
  an algorithm that given $x \in \R^{d}$ computes, with
  probability at least $1 - \delta$,
  an additive $\epsilon$-approximate smallest eigenvector
  $\what{v}$ of $\nabla^2 f(x)$ in time
  $ O \left( \T \left( 1 + \log ( \Dim / \delta ) \sqrt{ \SmGrad / \epsilon } \right) \right)$.
\end{corollary}

\section{Two structured non-convex problems}

With our preliminary results established, in this section we turn to two methods
that form the core of our approach. Roughly, our overall algorithm will be to
alternate between finding directions of negative curvature of $f$ and solving
structured sub-problems that are \emph{nearly} convex, meaning that the
smallest eigenvalue of the Hessian has a lower bound $-\gamma$, $\gamma > 0$,
where $\gamma \ll \SmGrad$. We turn to each of these pieces in turn.

\subsection{Accelerated gradient descent for almost convex functions}
\label{sec:almost-convex}

The second main component of our general accelerated method is a procedure for
finding stationary points of smooth non-convex functions that are not
\emph{too} non-convex. By not too non-convex, we mean
$\gamma$-almost convexity, as in Def.~\ref{def:strconvex}, that is,
that
\begin{equation*}
  f(y) \ge f(x) + \nabla f(x)^T (y - x) - \frac{\gamma}{2} \norm{x - y}^2
  ~~ \mbox{for~all~} x, y,
\end{equation*}
where $\gamma \ge 0$. The next procedure applies to such almost convex
functions, and builds off of a suggestion of \citet{NesterovGradSmall2012} to
use regularization coupled with accelerated gradient descent to improve
convergence guarantees for finding a stationary point of $f$.  The idea, as
per Lemma~\ref{lem:strong-convexity}, is to add a regularizing term of the
form $\gamma \norms{z - z_0}^2$ to make the $\gamma$-almost convex function
$f$ become $\gamma$-strongly convex. As we describe in the sequel, we solve a
sequence $j = 1, 2, \ldots$ of such proximal sub-problems
\begin{equation}
  \label{eqn:g-model-almost-convex}
  \minimize_z ~ g_j(z) \defeq f(z) + \gamma \norm{z - z_j}^2
\end{equation}
quickly using accelerated gradient descent. Whenever $\gamma \ll \SmGrad$, the
regularized model $g_j$ of $f$ has better fidelity to $f$ than the model
$f(z) + \frac{\SmGrad}{2} \norm{z - z_j}^2$ (which is essentially what gradient
descent attempts to minimize), allowing us to make greater progress in finding
stationary points of $f$. We now present the \callACAGD{} procedure.
\begin{algorithm}[H]
\begin{algorithmic}[1]\label{alg:ACAGD}
\Function{\ACAGD}{$f$, $z_{1}$, $\epsilon$, $\gamma$, $\SmGrad$} 
\For{$j = 1, 2, \ldots$}
\If{$\| \nabla f(z_{j}) \| \le \epsilon$} \Return{$z_{j}$}
\EndIf
\State Let $g_j(z) = f(z) + \gamma \norm{z - z_j}^2$ as in
model~\eqref{eqn:g-model-almost-convex}.
\State $\epsilon' = \epsilon \sqrt{ \gamma/(50 (\SmGrad + 2\gamma)) }$
\State $z_{j+1} \gets$ \callAGD{$g_j$, $z_j$, $\epsilon'$, $\SmGrad$, $\gamma$} \label{line:call-agd}
\EndFor
\EndFunction
\end{algorithmic}
\end{algorithm}

Recalling the definition $\DeltaF \ge f(z_1) - \inf_x f(x)$, we have the
following convergence guarantee.
\begin{lemma}
  \label{lem:almost-convex}
Let \func be be 
  $\min\{\StrConv, 0\}$-almost convex and
  $\SmGrad$-smooth. 
Let
  $\gamma \ge \StrConv$ and let $0 < \gamma \le \SmGrad$. Then
  \callACAGD{$f$,$z_1$,$\epsilon$,$\gamma$,$L_1$} returns a vector $z$ such
  that $\|\nabla f(z) \| \leq \epsilon$ and
  \begin{equation}
    \label{eq:almost_convex_dist}
    f(z_1) - f(z) \geq
    \min\left\{\gamma \|z - z_1\|^2 , \frac{\epsilon}{\sqrt{10}} \|z - z_1\| \right\}
  \end{equation}
  in time 
  \begin{flalign}\label{eq:almost-convex-complexity}
    O \left(   \T   \left(\sqrt{\frac{\SmGrad}{\gamma}}
        + \frac{\sqrt{\gamma \SmGrad}}{\epsilon^2}   (f(z_{1}) - f(z)) 
      \right)
       \log \left( 2 + \frac{\SmGrad^3 \DeltaF }{\gamma^2 \epsilon^2 } \right)    \right).
  \end{flalign}
\end{lemma}
Before providing the proof, we remark that the runtime
guarantee~\eqref{eq:almost-convex-complexity} is an improvement over the
convergence guarantees of standard gradient descent---which scale as
$O(\TGrad \DeltaF \SmGrad \epsilon^{-2})$---whenever $\gamma \ll \SmGrad$.

\begin{proof}
  Because $f$ is $-\StrConv$-strongly convex and $\gamma \geq \StrConv$,
  Lemma~\ref{lem:strong-convexity} guarantees that $g_j$ is $\gamma$-strongly
  convex. This strong convexity also guarantees that $g_j$ has a unique
  minimizer, which we denote $z^*_j$.

  Let $j_*$ be the time at which the routine terminates (we set $j_* =
  \infty$ if this does not occur; our analysis addresses this case).
  Let $j \in [1, j_*) \cap \N$ be arbitrary. We have by
  Line~\ref{line:call-agd} and Lemma~\ref{lem:AGD} (recall that $g_j$ is convex and
  $\SmGrad + 2\gamma$ smooth) that
  $ \norm{\nabla g_j(z_{j + 1})}^2 \leq \frac{\epsilon^2 \gamma}{\SmGrad +
    2\gamma}$.  Moreover, because $j < j_*$, we have
  $\| \nabla g_j(z_j) \| = \|\nabla f(z_j) \| \geq \epsilon$ by our
  termination criterion and definition~\eqref{eqn:g-model-almost-convex} of
  $g_j$. Consequently,
  $\norm{\nabla g_j(z_{j+1})}^2 \le \frac{\gamma}{\SmGrad + 2\gamma}
  \norm{\nabla g_j(z_j)}^2$, and applying Lemma~\ref{lem:grad:bounds} to the
  $(\SmGrad + 2 \gamma)$-smooth and $\gamma$-strongly convex function $g_j$
  yields that
  \begin{equation*}
    g_j(z_{j + 1}) - g_j(z_j^*)
    \leq \frac{1}{2\gamma} \| \nabla {g_j(z_{j + 1})}\|^2
    \leq \frac{1}{2(\SmGrad + 2\gamma)} \|\nabla{g_j(z_j)}\|^2
    \leq g_j(z_j) - g_j(z_j^*) .
  \end{equation*}
  Thus we have $g_j(z_{j + 1}) \leq g_j(z_j)$ and 
  \begin{equation*}
    f(z_{j + 1})
    = g_j(z_{j + 1}) - \gamma \norm{z_{j + 1} - z_j}^2
    \le g_j(z_j) - \gamma \norm{z_{j + 1} - z_j}^2
    = f(z_j) - \gamma \|z_{j + 1} - z_j\|^2.
  \end{equation*}
  Inducting on the index $j$, we have
  \begin{equation}
    \label{eq:almost_convex_progress}
    -\DeltaF \le
    f(z_{j_*}) - f(z_1) \leq - \gamma \sum_{j = 1}^{j_* - 1 } \|z_{j + 1} - z_j\|^2.
  \end{equation}

  Equation~\eqref{eq:almost_convex_progress} shows that to bound the number of
  iterations of the algorithm it suffices to lower bound the differences
  $\|z_{j + 1} - z_j\|$. Using the condition
  $ \| \nabla g_j(z_{j + 1}) \| \leq \epsilon \sqrt{\frac{\gamma}{50 (L_1 +
      2\gamma)}} \le \frac{1}{10} \epsilon$, we have
  \begin{equation*}
    \|z_{j + 1} - z_j\| = \frac{1}{2 \gamma} \| \nabla f(z_{j + 1}) - \nabla g_j(z_{j + 1}) \|
    \geq \frac{1}{2 \gamma} \left( \|\nabla f(z_{j + 1})\| - \frac{\epsilon}{10} \right),
  \end{equation*}
  where the inequality is a consequence of the triangle inequality.  By
  our termination criterion, we know that if $j + 1 < j_*$ then
  $\| \nabla f(z_{j + 1})\| \geq \epsilon$ and therefore
  $\|z_{j + 1} - z_j\| \geq \frac{9 \epsilon}{20 \gamma} \ge
  \frac{\epsilon}{\gamma \sqrt{5}}$. Substituting this bound into
  \eqref{eq:almost_convex_progress} yields
  \begin{equation*}
    -\DeltaF \le f(z_{j_*}) - f(z_1)
    \leq - \gamma \sum_{j = 1}^{j_* - 2} \|z_{j + 1} - z_j\|^2
    \leq - (j_* - 1) \cdot \frac{\epsilon^2}{5 \gamma}
  \end{equation*}
  and therefore
  \begin{equation}
    \label{eq:almost_convex_iter_bound}
    j_* \leq 1 + \frac{5 \gamma}{\epsilon^2} [f(z_1) - f(z_{j_*})]
    \le 1 + \frac{5 \gamma}{\epsilon^2} \DeltaF.
  \end{equation}

  Note that the method calls \callAGD{} (Line~\ref{line:call-agd}) with accuracy parameter
  $\epsilon' = \epsilon \sqrt{\gamma/(50(\SmGrad+2\gamma))}$; using
  $\gamma \le \SmGrad$ we may apply Lemma~\ref{lem:AGD} to bound
  the running time of each call by
  \begin{equation*}
    O\bigg(   \T  \bigg(1 + \sqrt{\frac{L_1 + 2\gamma}{\gamma}} 
    \log \frac{4 (L_1 + 2\gamma)^2 \DeltaF 
        }{ \gamma (\epsilon')^2} \bigg) \bigg)
    = O\left(\T \sqrt{\frac{L_1}{\gamma}} \log \left( 2 + \frac{\SmGrad^3 \DeltaF }{\gamma^2 \epsilon^2 } \right) \right).
  \end{equation*}
  The method \callACAGD{} performs at most $j^*$ iterations
  (Eq.~\eqref{eq:almost_convex_iter_bound}), and combining the preceding
  display with this iteration bound yields the running
  time~\eqref{eq:almost-convex-complexity}.

  All that remains is to prove the progress
  bound~\eqref{eq:almost_convex_dist}. By application
  of the triangle inequality and Jensen's inequality, we have
  \begin{equation*}
    \|z_{j^*} - z_1\|^2
    \leq \left(\sum_{j = 1}^{j^* - 1} 
      \|z_{j + 1} - z_j\|
    \right)^2 
    \leq j^* \cdot \sum_{j = 1}^{j^* - 1} 
    \|z_{j + 1} - z_j\|^2.
  \end{equation*}
  Combing this bound with the earlier progress
  guarantee~\eqref{eq:almost_convex_progress} yields
  $ f(z_1) - f(z_{j^*}) \geq \frac{\gamma}{j^*} \|z_j^* - z_1\|^2$, and since by~\eqref{eq:almost_convex_iter_bound}
  either $j^* \leq 1$ or
  $j^* \leq 10 \frac{\gamma}{\epsilon^2} [f(z_1) - f(z_{j^*})]$ the result
  follows.
\end{proof}

\subsection{Exploiting negative curvature}
\label{sec:negative-curvature}

Our first sub-routine either declares the problem locally ``almost convex'' or finds a direction of $f$ that has
negative curvature, meaning a direction $v$ such that
$v^T \nabla^2 f(x)v < 0$. The idea to make progress on $f$ by moving in
directions of descent on the Hessian is of course well-known, and relies on
the fact that if at a point $z$ the function $f$ is ``very''
non-convex, \ie $\lambda_{\min}(\nabla^2 f(z)) \le -\alpha / 2$ for some
$\alpha > 0$, then we can reduce the objective significantly (by a constant
fraction of $L_2^{-2} \alpha^3$ at least) by taking a step in a direction of 
negative curvature. Conversely, if $\lambda_{\min}(\nabla^2 f(z)) \ge -\alpha / 
2$, the function $f$ is ``almost convex'' in a neighborhood of $z$, suggesting 
that gradient-like methods on $f$ directly should be effective.  With this in 
mind, we present the routine $\callNCD{}$, which, given a function $f$, initial
point $z_1$, and a few additional tolerance parameters, returns a vector $z$ 
decreasing $f$ substantially by moving in Hessian-based directions.

\begin{algorithm}[H]
\begin{algorithmic}[1]\label{alg:NCD}
\Function{\NCD}{$z_1$, $f$, $\SmHess$, $\alpha$, $\DeltaF$, $\delta$}
\State Set $\delta' = \delta / \left(  1 + 12\SmHess^2 \DeltaF / \alpha^3 \right)$
\For {$j = 1, 2, \ldots$}
\State \label{line:find-eigs}
 Find a vector $v_j$ such that $\|v_j\| = 1$ and, with probability at least $1-\delta'$,
$$
\lambda_{\min}(\nabla^2 f(z_j))  \ge v_j^T \nabla^2 f(z_j) v_j - \alpha / 2
$$ 
 \hspace{30pt} using a leading eigenvector computation 
 \Comment{see Corollary~\ref{cor:smallest_evec_computation}}
\If{ $v_j^T \nabla^2 f(z_j) v_j\le - \alpha/2$}
\Comment{Make at least $\alpha^3/12\SmHess^2$ progress}
$$z_{j+1} \gets z_j - \frac{ 2| v_j^T \nabla^2 f(z_j) v_j | } {\SmHess }\mathop{\mathrm{sign}}(v_j^T\nabla f(z_j))v_j$$
\Else  
\Comment{w.h.p.,  $\lambda_{\min}(\nabla^2 f(z_j)) \geq -\alpha$}
\State \Return{ $z_{j}$ } \EndIf
\EndFor
\EndFunction
\end{algorithmic}
\end{algorithm}

We provide a formal guarantee for the method \callNCD{} in the following
lemma.  Before stating the lemma, we recall that $f : \R^d \to \R$ has
$\SmHess$-Lipschitz Hessian and that $\DeltaF \ge f(z_1) - \inf_x f(x)$.

\begin{lemma}\label{lem:NCD-iterations}
  Let the function \func be $\SmGrad$-smooth and have $\SmHess$-Lipschitz
  continuous Hessian, $\alpha > 0$, $0<\delta<1$ and $z_1 \in \R^d$. If we
  call \callNCD{$z_1$, $f$, $\SmHess$, $\alpha$, $\DeltaF$, $\delta$} then the
  algorithm terminates at iteration $j$ for some
  \begin{flalign}
    \label{eq:very-non-convex-it}
    j \le 1+\frac{12 \SmHess^2 (f(z_1)-f(z_j))}{\alpha^3}
    \le 1+\frac{12 \SmHess^2 \DeltaF}{\alpha^3},
  \end{flalign}
  and with probability at least $1- \delta$
  \begin{flalign}
    \label{eq:small-terminating-eig}
    \lambda_{\min}(\nabla^2 f(z_{j}))  \ge -\alpha.
  \end{flalign}
  Furthermore, each iteration requires time at most
   \begin{flalign}
    \label{eq:NCD-iteration-cost}
    O \left( \T\left[ 1+  \sqrt{ \frac{\SmGrad}{\alpha} } \log \left( \frac{\Dim}{\delta} 
          \left(  1 + 12\frac{\SmHess^2 \DeltaF}{\alpha^3} \right) \right) \right]\right).
  \end{flalign}
\end{lemma}

\begin{proof}
  Assume that the method has not terminated at iteration $k$. Let
  \begin{equation*}
    \eta_k = \frac{ 2| v_k^T \nabla^2 f(z_k) v_k | } {\SmHess }\mathop{\mathrm{sign}}(v_k^T\nabla f(z_k))
  \end{equation*}
  denote the step size used at iteration $k$, so that
  $z_{k + 1} = z_k - \eta_k v_k$ as in Line 5. By the $\SmHess$-Lipschitz
  continuity of the Hessian, we have
  \begin{equation*}
    | f(z_k - \eta_k v_k)- f(z_k) + \eta_k v_k^T \nabla f(z_k)  - \frac{1}{2}\eta_k^2  v_k^T \nabla^2 f(z_k) v_k |
    \le \frac{\SmHess}{6} \| \eta_k v_k \|^{3}.
  \end{equation*}
  Noting that $\eta_k v_k^T \nabla f(z_k) \ge 0$ by construction, we rearrange
  the preceding inequality to obtain
  \begin{equation*}
    f(z_{k+1}) - f(z_k) \le \frac{\eta_k^2}{2} \left( \frac{\SmHess}{3} | \eta_k | + v_k^T \nabla^2 f(z_k) v_k \right)
    = -\frac{2 |v_k^T \nabla^2 f(z_k) v_k|^3}{3\SmHess^2}
    \stackrel{(i)}{\le} -\frac{\alpha^3}{12 \SmHess^2} \ ,
  \end{equation*}
  where inequality~$(i)$ uses that $|v_k^T \nabla^2 f(z_k) v_k| > \alpha/2$, 
  as the stopping criterion has not been met. Telescoping the above equation
  for $k= 1, 2, \ldots ,j-1$, we conclude that at the final iteration
  $$
  \DeltaF \ge f(z_1) - f(z_j) \ge \frac{\alpha^3}{12 \SmHess^2}(j-1) \ ,
  $$
  which gives the bound~\eqref{eq:very-non-convex-it}.

  We turn to inequality~\eqref{eq:small-terminating-eig}.  Recall the
  definition of
  $\delta' = \frac{\delta}{1 + 12 \SmHess^2 \DeltaF / \alpha^3}$, which
  certainly satisfies $\delta' \le \delta / j$ if $j$ is the final iteration
  of the algorithm (as the bound~\eqref{eq:very-non-convex-it} is
  deterministic). Now, at the last iteration, we have by definition of the
  final iterate that $v_j^T \nabla^2 f(z_j) v_j \ge -\frac{\alpha}{2}$, and
  thus, if $v_j$ is an additive $\alpha/2$-approximate smallest eigenvector,
  we have
  $\lambda_{\min}(\nabla^2 f(z_j)) \ge v_j^T \nabla^2 f(z_j) v_j - \alpha$.
  Applying a union bound, the probability that the approximate eigenvector
  method fails to return an $\alpha/2$-approximate eigenvector
  in any iteration is bounded by $\delta' j \le \delta$,
  giving the result.
  
  Finally, equation~\eqref{eq:NCD-iteration-cost} is immediate by
  Corollary \ref{cor:smallest_evec_computation}.
\end{proof}

\section{An accelerated gradient method for non-convex optimization}
\label{sec:main-theorem}

Now that we have provided the two subroutines \callNCD{} and \callACAGD{},
which (respectively) find directions of negative curvature and solve nearly
convex problems, we combine them carefully to provide an accelerated gradient
method for smooth non-convex optimization.  The idea behind our $\callACCNC{}$
is as follows. At the beginning of each iteration $k$ we use $\callNCD{}$ to make
progress until we reach a point $\what{x}_k$ where the function is almost
convex (Def.~\ref{def:strconvex}) in a neighborhood of the current
iterate. For a parameter $\alpha \ge 0$, we define the convex penalty
\begin{flalign}
  \label{eq:rhodef}
  \rho_\alpha(x) \defeq \SmGrad \hinge{\| x \| -  \rad}^2,
\end{flalign}
where $\hinge{t} = \max\{t, 0\}$. We then modify the function 
$f(x)$ by
adding the penalty $\rho_\alpha$ and defining
\begin{flalign*}
  f_k(x) = f(x) + \rho_\alpha(x-\what{x}_k) .
\end{flalign*}
The function $f_k(x)$ is globally almost convex, as we show in
Lemma~\ref{lem:f-negatively-convex} to come, so that the method $\callACAGD{}$
applied to the function $f_k(x)$ quickly reduces the objective $f$. We trade
between curvature minimization and accelerated gradient using the parameter
$\alpha$ in the definition~\eqref{eq:rhodef} of $\rho$, which governs
acceptable levels of non-convexity. By carefully choosing $\alpha$, the
combined method has convergence rate $\wt{O}( \epsilon^{-7/4} )$, which we
we prove in Theorem~\ref{thm:complexity-of-ACCNC}.

\begin{algorithm}[H]
\caption{Acceleration of smooth non-linear optimization}\label{alg:ACCNC}
\begin{algorithmic}[1]
\Function{\ACCNC}{$x_1$, $f$, $\epsilon$, $\SmGrad$, $\SmHess$, $\alpha$, $\DeltaF$, $\delta$}
\State Set $K := \ceil{1+\DeltaF(12 \SmHess^2/\alpha^3 + 
\sqrt{10}\SmHess/(\alpha
  \epsilon))}$ and $\delta'' := \frac{\delta}{ K }$ \label{line:ACCCNC-K}
\For {$k = 1, 2, \ldots$}
\If{$\alpha < \SmGrad$}
\State $\what{x}_{k} \gets$ \callNCD{$x_k$, $f$, $\SmHess$, $\alpha$, $\DeltaF$, $\delta''$}
\label{line:ACCNC-callNCD}
\Else
\State $\what{x}_{k} \gets x_{k}$
\label{line:what-same}
\EndIf
\If{$\| \nabla f(\what{x}_{k}) \| \le \epsilon$}
\State\Return{ $\what{x}_{k}$ } \Comment{guarantees w.h.p., $\lambda_{\min}(\nabla ^2 f(\what{x}_{k})) \ge -2\alpha$ }

\EndIf
\State Set $f_k(x) = f(x) + \SmGrad\left(\hinge{\|x-\what{x}_k\| -\alpha/\SmHess }\right)^2$
\State $x_{k+1} \gets$ \callACAGD{$f_k$, $\what{x}_k$, $\epsilon/2$, $3\alpha$, $5 \SmGrad$} \label{line:call-acagd}
\EndFor
\EndFunction
\end{algorithmic}
\end{algorithm}

\subsection{Preliminaries: convexity and iteration bounds}

Before coming to the theorem giving a formal guarantee for \callACCNC{}, we
provide two technical lemmas showing that the internal subroutines are
well-behaved.  The first lemma confirms that the
regularization technique~\eqref{eq:rhodef} transforms a locally almost convex
function into a globally almost convex function
(Def.~\ref{def:strconvex}), so we can efficiently apply $\callACAGD{}$ to
it.

\begin{lemma}\label{lem:f-negatively-convex}
  Let $f$ be $\SmGrad$-smooth and have $\SmHess$-Lipschitz continuous
  Hessian. Let $x_0\in\Reals^\Dim$ be such that
  $\nabla ^2 f(x_0) \succeq -\alpha I$ for some $\alpha \ge 0$. The function
  $f_\alpha(x) \defeq f(x)+\rho_\alpha(x-x_0)$ is $3 \alpha$-almost convex and
  $5 \SmGrad$-smooth.
\end{lemma}
\begin{proof}
  It is clear that $\rho = \rho_\alpha$ is convex, as it is an increasing
  convex function of a positive argument~\cite[Chapter 3.2]{BoydVa04}.
  We claim that $\rho$ is $4\SmGrad$-smooth. Indeed, the gradient
  \begin{equation*}
    \nabla \rho (x) = 2\SmGrad \frac{x}{\|x\|}
    \hinge{\norm{x} - \rad}
  \end{equation*}
  is continuous by inspection and differentiable except at $\norm{x} = \rad$.
  For $\|x\| < \radv$, we have $\nabla ^2 \rho(x) = 0$, and for $\|x\| > \radv$ we have
  \begin{flalign}\label{eq:rhohess}
    \nabla^2 \rho(x) = 2 \SmGrad \left( I +  \frac{\alpha}{\SmHess}  \left(\frac{x x^T}{\| x \|^3 } -  \frac{I}{\| x  \|} \right) \right),
  \end{flalign}
  which satisfies $0 \preceq \nabla^2 \rho(x) \preceq 4\SmGrad I$ for all
  $x$. As $\nabla \rho(x)$ is continuous,
  we conclude that $\rho$ is
  $4\SmGrad$-smooth. The $\SmGrad$-smoothness of $f$ then implies
  that the sum $f(x)+\rho(x-x_0)$ is $5\SmGrad$ smooth.
	
  To argue almost convexity of $f + \rho$, we show that
  $\nabla^2 f(x) + \nabla^2 \rho(x-x_0) \succeq -3\alpha I$ almost everywhere,
  which is equivalent to Definition~\ref{def:strconvex} when the gradient
  is continuous. For $\|x-x_0\| < 2\radv$, we have by Lipschitz
  continuity of $\nabla^2 f$ that
  \begin{equation*}
    \nabla^2f (x) \succeq \nabla^2 f(x_0) - \SmHess \|x-x_0\| I \succeq -3 \alpha I,
  \end{equation*}
  which implies the result because $\rho$ is convex. For $\|x-x_0\| > 2\radv$,
  inspection of the Hessian \eqref{eq:rhohess} shows that
  $\nabla^2 \rho(x-x_0) \succeq \SmGrad I$. Since
  $\nabla^2 f(x) \succeq -\SmGrad I$ almost everywhere by the
  $\SmGrad$-smoothness of $f$, we conclude that
  $\nabla^2 f(x) + \nabla^2 \rho(x-x_0) \succeq 0$ whenever $\nabla^2 f(x)$
  exists.
\end{proof}

The next lemma provides a high probability guarantee on the correctness and
number of iterations of \callACCNC{}. (There is randomness in the eigenvector
computation subroutine invoked within \callNCD{}.) As always, we let
$\DeltaF \ge f(x_1) - \inf_x f(x)$.

\begin{lemma}\label{lem:iterations-of-ACCNC}
  Let $f$ be $\SmGrad$-smooth with $\SmHess$-Lipschitz continuous Hessian,
  $\epsilon > 0$, $\delta \in (0, 1)$, and $\alpha \in [0, \SmGrad]$.
  Then with probability at least $1 - \delta$, the method
  \callACCNC{$x_1$, $f$, $\epsilon$,
    $\SmGrad$, $\SmHess$, $\alpha$, $\DeltaF$, $\delta$} terminates after 
    $t$ iterations with $\|\nabla f(\what{x}_{t})\| \le \epsilon$, where
    $t$ satisfies
    \begin{equation}
      \label{eqn:acc-non-convex-iterations-both}
      t \le \begin{cases}
        2 + \DeltaF \left( \frac{12 \SmHess^2}{\alpha^3} + \frac{\sqrt{10} 
            \SmHess }{ \alpha \epsilon } \right)
        & \mbox{if}~ \alpha < \SmGrad \\
        2 + \DeltaF \frac{16 \SmGrad}{3 \epsilon^2}
        & \mbox{if~} \alpha = \SmGrad
      \end{cases}
    \end{equation}
   Further, $\lambda_{\min}(\nabla^2 f(\what{x}_{k})) \ge -2\alpha$ for all 
   $k \le t$.
 \end{lemma}
\begin{proof}
  Before beginning the proof proper, we provide a quick bound on the size of
  the difference between iterates $\what{x}_k$ and $\what{x}_{k-1}$, which
  will imply progress in function values across iterations of
  Alg.~\ref{alg:ACCNC}.  In each iteration that the convergence criterion
  $\norm{\nabla f(\what{x}_k)} \le \epsilon$ is not met---that is, whenever
  $\norm{\nabla f(\what{x}_k)} > \epsilon$---we have that
 \begin{equation*}
   \epsilon \le \norm{\nabla f(\what{x}_k)}
   \stackrel{(i)}{\le} \| \nabla f_{k-1} (\what{x}_{k} ) \|
   + \| \nabla \rho (\what{x}_{k} - \what{x}_{k-1}) \|
   \stackrel{(ii)}{\le} \frac{\epsilon }{2} + 2\SmGrad
    \hinge{\norm{\what{x}_{k} - \what{x}_{k-1}} - \rad}.
  \end{equation*}
  In inequality~$(i)$ we used the triangle inequality and definition of
  $f_{k-1} = f + \rho(\cdot - x_{k-1})$ and inequality $(ii)$ used that the
  call to \callACAGD{} returns $\what{x}_k$ with
  $\norm{\nabla f_{k-1}(x_{k})} \le \epsilon/2$.  Rearranging yields
  \begin{flalign}
    \label{eqn:xnorm-bound}
    \frac{\epsilon }{4\SmGrad} \le
    \hinge{\norm{ \what{x}_{k} - \what{x}_{k-1}} - \rad} = \norm{ \what{x}_{k} - 
      \what{x}_{k-1}} - \rad,
  \end{flalign}
  where the equality is implied because $\epsilon > 0$.

  Now we consider two cases, the first the simpler
  case that $\alpha = \SmGrad$ is large enough that we never search
  for negative curvature, and the second that $\alpha < \SmGrad$ so that
  we find directions of negative curvature in the method.

  \paragraph{Case 1: large $\alpha$} In this case, we have that
  $\alpha = \SmGrad$, so that $x_k = \what{x}_k$ for all iterations $k$
  (Line~\ref{line:what-same} of the algorithm).  Assume that at iteration $k$
  that the algorithm has not terminated, so
  $\| \nabla f (\what{x}_{k} ) \| \ge \epsilon$. Then
  inequality~\eqref{eqn:xnorm-bound} gives
  $\frac{\epsilon}{4 \SmGrad} < \| \what{x}_{k} - \what{x}_{k-1} \|$.  By
  Lemma~\ref{lem:f-negatively-convex} we know that $f_k$ is $3 \SmGrad$-almost
  convex (Def.~\ref{def:strconvex}) and $5 \SmGrad$-smooth; therefore we may
  apply Lemma~\ref{lem:almost-convex} with $\gamma = 3 \alpha
  = 3 \SmGrad$ to lower bound
  the progress of the call to \callACAGD{} in Line~\ref{line:call-acagd}
  of Alg.~\ref{alg:ACCNC} to obtain
  \begin{flalign}
    f(\what{x}_{k-1}) - f(\what{x}_{k}) & \ge \min\left\{3 \SmGrad \| \what{x}_{k-1} - \what{x}_{k} \|^2 ,
      \frac{\epsilon}{\sqrt{10}} \| \what{x}_{k-1} - \what{x}_{k}\| \right\} \nonumber \\
    & \ge \min \left\{3\frac{\epsilon^2}{16 \SmGrad},
      \frac{ \epsilon^2 }{4 \sqrt{10} \SmGrad} \right\} \ge  \frac{\epsilon^2}{16 \SmGrad}.
  \end{flalign}
  Telescoping this display, we have for any iteration $s$ at which the algorithm
  has not terminated that
  \begin{equation*}
    \DeltaF \ge  \sum_{k = 2}^{s} f(\what{x}_{k-1}) - f(\what{x}_k) \ge (s - 1) \frac{3 \epsilon^2}{16 \SmGrad}
  \end{equation*}
  which yields the second case of the bound~\eqref{eqn:acc-non-convex-iterations-both}.
  The inequality $\nabla^2 f(\what{x}_j) \succeq -2 \alpha I$ holds trivially because
  $f$ is $\SmGrad$-smooth.

  \paragraph{Case 2: small $\alpha$} In this case, we assume
  that $\alpha < \SmGrad$. Let
  $K = \ceil{1 + \DeltaF (\frac{12 \SmHess^2}{\alpha^3} + \frac{\sqrt{10} \SmHess}{
    \alpha \epsilon})}$ and $\delta'' = \frac{\delta}{K}$ as in
  line~\ref{line:ACCCNC-K} of Alg.~\ref{alg:ACCNC}.
  By  Lemma~\ref{lem:NCD-iterations} and a union bound, with probability at
  least $1-\delta$, for all
  $k \le K$ the matrix inequality $\nabla^2 f(\what{x}_k) \succeq -2 \alpha I$ holds,
  so that we perform our subsequent analysis (for $k \le K$) conditional on this
  event without appealing to any randomness.

  Equation~\eqref{eqn:xnorm-bound} implies that at
  iteration $1 < k \le K$ exactly one of following three cases is true:
  \begin{enumerate}[(i)]
  \item The termination criterion $\| \nabla f (\what{x}_k) \| \le \epsilon$ holds
    and Alg.~\ref{alg:ACCNC} terminates.
  \item \callNCD{} (Line~\ref{line:ACCNC-callNCD})
    constructs $\what{x}_k \neq x_k$, and (i) fails.
  \item Neither (i) nor (ii) holds, and  $\| \what{x}_k - \what{x}_{k-1} \| \ge 
    \radv$. 
  \end{enumerate}
  We claim that in case (ii) or (iii), we have
  \begin{equation}
    \label{eqn:claimed-ACCNC-progress}
    f(\what{x}_{k-1}) - f(\what{x}_k) \ge \min\left\{\frac{\alpha \epsilon}{\SmHess \sqrt{10}},
      \frac{\alpha^3}{12 \SmHess^2}\right\}.
  \end{equation}
  Deferring the proof of claim~\eqref{eqn:claimed-ACCNC-progress}, we note that
  it immediately gives a quick proof of the result. 
  Assume, in order to obtain a contradiction that after $K$ iterations the algorithm has not terminated it follows that:
  \begin{equation*}
    \DeltaF \ge f(\what{x}_1) - f(\what{x}_K)
    = \sum_{k = 2}^{K} f(\what{x}_k) - f(\what{x}_{k+1})
    \ge (K - 1) \min\left\{\frac{\alpha \epsilon}{\SmHess \sqrt{10}},
      \frac{\alpha^3}{12 \SmHess^2} \right\},
  \end{equation*}
  Substituting for $K = \ceil{1 + \DeltaF (\frac{12 \SmHess^2}{\alpha^3} + \frac{\sqrt{10} \SmHess}{
    \alpha \epsilon})}$ as in
  line~\ref{line:ACCCNC-K} yields a contradiction and therefore the algorithm terminates after at most $K$ iterations which is the first case of the bound~\eqref{eqn:acc-non-convex-iterations-both}.
  
  Let us now prove the claim~\eqref{eqn:claimed-ACCNC-progress}. First,
  assume case (ii). Then
  \callNCD{} requires at least two iterations, so Lemma~\ref{lem:NCD-iterations}
  implies
  \begin{equation*}
    \frac{12 \SmHess^2 (f(x_{k})-f(\what{x}_k))}{\alpha^3} \ge 1.
  \end{equation*}
  Combining this with the fact that $f(x_k) \le f(\what{x}_{k-1})$ by the
  progress bound~\eqref{eq:almost_convex_dist} in
  Lemma~\ref{lem:almost-convex} (\callACAGD{} decreases function values),
  we have
   \begin{equation*}
     f(\what{x}_{k-1}) - f(\what{x}_{k})
     \ge f(x_{k})-f(\what{x}_k) \ge \frac{\alpha^3}{12 \SmHess^2}.
   \end{equation*}
   Let us now consider the case that (iii) holds. By
   Lemma~\ref{lem:f-negatively-convex} we know that the
   constructed function $f_k$ is $3\alpha$-almost
   convex (Def.~\ref{def:strconvex}) and $5 \SmGrad$-smooth, therefore we may
   apply Lemma~\ref{lem:almost-convex} with $\gamma = 3 \alpha$ to lower bound
   the progress of the entire inner loop of Alg.~\ref{alg:ACCNC} by
  \begin{align*}
    f(\what{x}_{k-1}) - f(\what{x}_{k})
    \ge \min\left\{\gamma \| \what{x}_{k-1} - x_{k} \|^2 ,
    \frac{\epsilon}{\sqrt{10}} \| \what{x}_{k-1} - x_{k}\| \right\}
    \ge \min \left\{ \frac{3 \alpha^3 }{L_{2}^2},
    \frac{ \alpha \epsilon}{ L_{2} \sqrt{10}} \right\}
  \end{align*}
  as desired.
\end{proof}

\subsection{Main result}

With Lemmas~\ref{lem:f-negatively-convex} and~\ref{lem:iterations-of-ACCNC} in
hand, we may finally present our main result.

\begin{theorem}\label{thm:complexity-of-ACCNC}
  Let \func be $\SmGrad$-smooth and have $\SmHess$-Lipschitz continuous
  Hessian.  Let
  \begin{equation*}
    \alpha =  \min \left\{ \SmGrad, \max \left\{ \epsilon^2 \DeltaF^{-1}, \epsilon^{1/2}\SmHess^{1/2} \right\} \right\}
  \end{equation*}
  and $\delta \in (0,1)$. Then with probability at least $1 - \delta$,
  \callACCNC{$x_1$, $f$, $\epsilon$, $\SmGrad$, $\SmHess$, $\alpha$,
    $\DeltaF$, $\delta$} returns a point $x$ that satisfies
  \begin{equation*}
    \|\nabla f(x)\| \le \epsilon
    ~\textrm{ and }~
    \lambda_{\min}(\nabla^2 f(x)) \ge -2 \epsilon^{1/2}\SmHess^{1/2}
  \end{equation*}
  in time
  \begin{flalign*}
    O\left( \T \left( \DeltaF \SmGrad^{1/2} \SmHess^{1/4} \epsilon^{-7/4} + 
    \DeltaF^{1/2} \SmGrad^{1/2} \epsilon^{-1} + 1 \right)
      \log\tau \right),
  \end{flalign*}
  where $\tau = 1 + 1/\epsilon + 1 / \delta  + \Dim + \SmGrad + \SmHess + 
  \DeltaF$.
\end{theorem}
\begin{proof}
  We split this proof into two cases: (I) small alpha when
  $\alpha < \SmGrad$ and hence $\alpha = \frac{\epsilon^2}{\DeltaF}$ or $\sqrt{\SmHess \epsilon}$,
  this is the non-trivial case requiring solution to a reasonably small accuracy; and (II)
  when $\alpha = \SmGrad$, when the algorithm is roughly equivalent to
  gradient descent (and $\epsilon$ is large enough that we do not require
  substantial accuracy).

  \paragraph{Case I: Small $\alpha$}
  We proceed in two steps. First, we bound the number of
  eigenvector calculations that \callNCD{} performs by providing
  a progress guarantee for each of them using Lemma~\ref{lem:NCD-iterations}
  and arguing that
  making too much progress is impossible. After this, we perform
  a similar calculation for the total number of gradient
  calculations throughout calls to \callACAGD{}, this time applying
  Lemma~\ref{lem:almost-convex}.

  We begin by bounding the number of eigenvector calculations.  When
  $\alpha < \SmGrad$, its definition implies
  $\epsilon \le \min\{\SmGrad^2/\SmHess,\DeltaF^{1/2}\SmGrad^{1/2}\}$.  Let
  $j_k^{*}$ be the total number of times the method $\callNCD{}$ invokes the
  eigenvector computation subroutine (Line~\ref{line:find-eigs} of \callNCD{})
  during iteration $k$ of the method \callACCNC{}, let $k^*$ denote the
  total number of iterations of \callACCNC{}, and define
  $q \defeq \sum_{k = 1}^{k^*} j_k^*$ as the total
  number of eigenvector computations.
  Then by telescoping the bound~\eqref{eq:very-non-convex-it}
  in Lemma~\ref{lem:NCD-iterations} and using that
  $f(x_{k}) \le f(\what{x}_{k-1})$ by the
  progress bound~\eqref{eq:almost_convex_dist} in
  Lemma~\ref{lem:almost-convex} (\callACAGD{} decreases function values), we have
  \begin{equation*}
    \sum_{k = 1}^{k^*} (j_k^* - 1)
    \le \sum_{k = 1}^{k^*}
    \frac{12 \SmHess^2}{\alpha^3}
    (f(x_k) - f(\what{x}_k))
    \le \sum_{k = 1}^{k^*}
    \frac{12 \SmHess^2}{\alpha^3}
    (f(\what{x}_{k-1}) - f(\what{x}_k))
    \le \frac{12 \DeltaF \SmHess^2}{\alpha^3}.
  \end{equation*}
  Substituting the bound on $k^*$ that Lemma~\ref{lem:iterations-of-ACCNC}
  supplies, we see that with probability at least $1-\delta$,
  \begin{flalign}
    \label{eq:total-eig-calls}
    q \le \frac{12 \DeltaF \SmHess^2}{\alpha^3}
    + k^*
    \le 1 + \DeltaF  \left( \frac{24 \SmHess^2}{\alpha^3} + \frac{\sqrt{10} 
        \SmHess }{ \alpha \epsilon } \right) \stackrel{(i)}{\le} 1 + 28 \DeltaF 
    \SmHess^{1/2} \epsilon^{-3/2},
  \end{flalign}
  where inequality $(i)$ follows by our construction that
  $\alpha \ge \epsilon^{1/2}\SmHess^{1/2}$.
  Inequality~\eqref{eq:total-eig-calls} thus provides a bound on the total
  number of fast eigenvector calculations we require.

  We use the bound~\eqref{eq:total-eig-calls}
  to bound the total cost of calls to \callNCD{}.
  The tolerated failure probability $\delta''$ defined in
  line~\ref{line:ACCCNC-K} satisfies
  \begin{equation*}
    \frac{1}{\delta''} = \frac{ 1 +\DeltaF(12 \SmHess^2/\alpha^3 + 
      \sqrt{10}\SmHess/(\alpha
      \epsilon)) }{\delta}
    \le \frac{ 1 + 16 \DeltaF \SmHess^{1/2} \epsilon^{-3/2} }{\delta},
  \end{equation*}
  so that $\log \frac{1}{\delta''} = O(\log \tau)$.
  By Lemma~\ref{lem:NCD-iterations}, Eq.~\eqref{eq:NCD-iteration-cost},
  the cost of each iteration during \callNCD{} is, using $\max\{ \epsilon^2 
  \DeltaF^{-1}, \sqrt{\epsilon \SmHess}\} = \alpha < \SmGrad$,
  at most
  \begin{equation*}
    O\left(\! \T \left[1+\sqrt{\frac{\SmGrad}{\alpha}} \log\left( 
    \frac{\Dim}{\delta''}   \left(  1 
    + 
    \frac{12\SmHess^2 \DeltaF}{\alpha^3} \right)\right)\right]\!\right) = 
    O\left( \T\frac{\SmGrad^{\half}}{(\SmHess 
    \epsilon)^{\frac{1}{4}}\vee(\epsilon\DeltaF^{\half})}  
    \log\tau \right) \!.
  \end{equation*}
  Multiplying this time complexity by $q$ as bounded in
  expression~\eqref{eq:total-eig-calls} gives that the total cost of the calls
  to \callNCD{} is
  \begin{flalign}
    \label{eq:total-cost-eigenvectors}
    O\left( \T \left( \DeltaF \SmGrad^{1/2} \SmHess^{1/4} \epsilon^{-7/4} 
    +\SmGrad^{1/2} \DeltaF^{1/2} \epsilon^{-1} \right) \log\tau \right) .
  \end{flalign}
	
  We now compute the total cost of calling $\callACAGD{}$. Using the time
  bound~\eqref{eq:almost-convex-complexity} of Lemma~\ref{lem:almost-convex},
  the cost of calling $\callACAGD{}$ in iteration $k$ with almost convexity
  parameter $\gamma = 3 \alpha$ is bounded by the sum of
  \begin{equation}
    \label{eqn:intermediate-accel-gd-comps}
    O\left(\T \sqrt{\frac{\SmGrad}{\gamma}} \log \tau \right)
    ~~ \mbox{and} ~~
    O\left(\T \frac{\sqrt{\gamma \SmGrad}}{\epsilon^2}
      [f_{k}(x_{k}) - f_{k}(x_{k+1})]\log\tau \right).
  \end{equation}
  We separately bound the total computational cost of each of the
  terms~\eqref{eqn:intermediate-accel-gd-comps}.

  Using the bound
  $k^* \le 1 + \DeltaF \frac{16 \SmHess^{1/2}}{\epsilon^{3/2}}$ as in
  expression~\eqref{eq:total-eig-calls} for the total number of iterations of
  Alg.~\ref{alg:ACCNC}, we see that the first of the time
  bounds~\eqref{eqn:intermediate-accel-gd-comps} yields identical total cost
  to the eigenvector computations~\eqref{eq:total-cost-eigenvectors}, because
  $\gamma^{-\half} = O(\alpha^{-\half}) = O(1 / \sqrt[4]{\epsilon \SmHess})$.
  Thus we consider the second term in
  expression~\eqref{eqn:intermediate-accel-gd-comps}.
  Using the fact that $[f_k(x_k) - f_{k}(x_{k+1})] \le [f(x_k) - f(x_{k+1})]$
  by
  definition of $x_{k+1}$ and the method \callACAGD{}, we telescope to find
  \begin{equation*}
    \sum_{k = 1}^{k^*} [f_{k}(x_{k}) - f_{k}(x_{k+1})]
    \le \sum_{k = 1}^{k^*} [f(x_{k}) - f(x_{k+1})]
    \le \DeltaF.
  \end{equation*}
  Noting that by assumption that the almost convexity parameter
  $\gamma = 3\alpha$, we have
  $\sqrt{\gamma/3} = \sqrt{\alpha} \le \SmHess^{1/4} \epsilon^{1/4} +
  \epsilon\DeltaF^{-1/2}$, telescoping the second
  term of the bound~\eqref{eqn:intermediate-accel-gd-comps}
  on the cost of \callACAGD{}
  immediately gives the total computational cost bound
  \begin{equation*}
    O\left( \T \frac{\DeltaF \SmGrad^{1/2}}{\epsilon^2} \left( \SmHess^{1/4} 
        \epsilon^{1/4} + \frac{\epsilon}{\sqrt{\DeltaF}} \right)  \log\tau
    \right)
   \end{equation*}
   over all calls of \callACAGD{}. This is evidently our desired result that
   the total computational cost when $\alpha < \SmGrad$
   is~\eqref{eq:total-cost-eigenvectors}.

   \paragraph{Case II: Large $\alpha$}
   When $\alpha = \SmGrad$, the algorithm becomes roughly equivalent to
   gradient descent, because \callNCD{} is not required, so that we need only
   bound the total computational cost of calls to \callACAGD{}.
   The bound~\eqref{eqn:intermediate-accel-gd-comps} on the computational
   effort of each such call again applies, and noting that
   $\SmGrad = \alpha = 3 \gamma$ in this case, we replace the
   bounds~\eqref{eqn:intermediate-accel-gd-comps}
   with the two terms
   \begin{equation*}
     O (\T \log \tau)
     ~~ \mbox{and} ~~
     O \left(   \T  \frac{\SmGrad}{\epsilon^2}  [f(x_{k}) - f(x_{k+1})] 
     \log \tau \right).
 \end{equation*}
 As in Case I, we may telescope the second time bound to obtain
 total computational effort $O(\T (1 + \DeltaF \frac{\SmGrad}{\epsilon^2}) \log
 \tau)$,
 while applying the iteration bound~\eqref{eqn:acc-non-convex-iterations-both}
 of Lemma~\ref{lem:iterations-of-ACCNC}
 to the first term similarly yields the bound
 $O(\T (1 + \DeltaF \frac{\SmGrad}{\epsilon^2}) \log \tau)$ on the total
 computational
 cost.
 To conclude the proof we  
 observe that 
 $\alpha = \SmGrad$ implies $\SmGrad \le \max\{(\epsilon 
   \SmHess)^{1/2},\epsilon^2/\DeltaF\}$ and that $\SmGrad \le  
   \max\{(\epsilon 
   \SmHess)^{1/4}\SmGrad^{1/2},\epsilon^2/\DeltaF\}$. Therefore, 
   \begin{flalign*}
     O\left( \T  \left( 1 + \DeltaF \frac{\SmGrad}{\epsilon^2} \right) \log \tau 
     \right) = O\left( \T \left(1 + 
         \frac{\DeltaF \SmGrad^{1/2} \SmHess^{1/4}}{\epsilon^{7/4}} \right) 
       \log\tau \right),
   \end{flalign*}
   which gives our desired total time.
\end{proof}

We provide a bit of discussion to help explicate this result.
Much of the complication in the statement of
Theorem~\ref{thm:complexity-of-ACCNC}
is a consequence
of our desire for generality in possible parameter values. In common
settings in which points reasonably close to stationarity are desired---when
the accuracy $\epsilon$ is small enough---we may simplify the
theorem substantially, as the following corollary demonstrates.
\begin{corollary}
  Let the conditions of Theorem~\ref{thm:complexity-of-ACCNC} hold,
  and in addition assume that
  $\epsilon \le \sqrt[3]{\DeltaF^2 / \SmHess}$.
  Then the total computational cost of Alg.~\ref{alg:ACCNC} is
  at most
  \begin{flalign*}
    \wt{O}\left( \T  \DeltaF \frac{\SmGrad^{1/2} \SmHess^{1/4}}{\epsilon^{7/4}} \right).
  \end{flalign*}
\end{corollary}

To elucidate the relative importance of acceleration in the approximate
eigenvector or gradient descent computation in \callACCNC{}, we may also
consider replacing them with (respectively) the power method (rather than the
Lanczos method) or standard gradient descent.  We first consider the
accelerated (approximate) eigenvector routine.  With probability at least
$1 - \delta$, the power method finds an $\alpha$-additive approximate maximum
or minimum eigenvector of the matrix $\nabla^2 f(x) \in \R^{d \times d}$, with
operator norm bounded as $\norm{\nabla^2 f(x)} \le \SmGrad$, in time
$O(\frac{\SmGrad}{\alpha} \log \frac{d}{\delta})$ (compare this with
Corollary~\ref{cor:smallest_evec_computation}). In this case, substituting
$\alpha \asymp \epsilon^{4/9}$, rather than $\alpha \asymp \epsilon^{1/2}$ in
Theorem~\ref{thm:complexity-of-ACCNC}, and mimicking the preceding proof
yields total time complexity of order $\epsilon^{-16/9} \ll \epsilon^{-2}$, ignoring all
other problem-dependent constants. That is, non-accelerated eigenvector
routines can still yield faster than $\epsilon^{-2}$ rates of convergence.

Conversely, it appears that accelerated gradient descent is more central to
our approach.  Indeed, the term involving $\sqrt{\gamma \SmGrad}$ in the
bound~\eqref{eqn:acc-non-convex-iterations-both} is important, as it allows us
to carefully trade ``almost'' convexity $\gamma$ with accuracy $\epsilon$ to
achieve fast rates of convergence. Replacing the accelerated gradient descent
method with gradient descent in \callACAGD{} eliminates the possibility for
such optimal trading.  Of course, our procedure would still produce output
with the second order guarantee
$\nabla^2 f(x) \succeq -2 \epsilon^{1/2}\SmHess^{1/2} I_{d \times d}$.

\section{Accelerated (linear) convergence to local minimizers of strict-saddle
  functions}\label{sec:strict-saddle}

In this section, we show how to apply \callACCNC{} and
Theorem~\ref{thm:complexity-of-ACCNC} to find \emph{local minimizers} for
generic non-pathological \\ smooth optimization problems with linear rates of
convergence. Of course, it is in general NP-hard to even check if a point is a
local minimizer of a smooth nonconvex optimization problem~\cite{MurtyKa87,
  Nesterov00}, so we require a few additional assumptions in this case. In
general, second-order stationary points need not be local minima;
consequently, we consider \emph{strict-saddle} functions, which are functions
such that all eigenvalues of the Hessian are non-zero at all critical points,
so that second-order stationary points are indeed local minima. Such
structural assumptions have been important in recent work on first-order
methods for general smooth minimization~\cite{lee2016gradient, ge2015escaping,
  SunQuWr15}, and in a sense ``random'' functions generally satisfy these
conditions (cf.\ the discussion of Morse functions in the
book~\cite{AdlerTa09}).
To make our discussion formal, consider the following quantitative definition.
\begin{definition}
  A twice differentiable function \func is
  \emph{$(\varepsilon, \sigma_{-}, \sigma_{+})$-strict-saddle} if for any
  point $x$ such that $\norm{\nabla f(x)} \le \varepsilon$,
  $\lambda_{\min}(\nabla^2 f(x)) \in (-\infty, \sigma_{-}]\cup [\sigma_{+},
  \infty)$.
\end{definition}

Some definitions of strict-saddle include a radius
$R$ bounding the distance between any point $x$ satisfying $\norm{\nabla 
f(x)} 
\le \varepsilon$ and $\lambda_{\min}(\nabla^2 f(x)) \ge \sigma_{+}$ and a 
local minimizer $x^+$ of $f$, and they assume that $f$ is $\sigma_{+}$-strongly convex 
in a ball of radius $2R$ around any local minimizer. Our assumption
on the Lipschitz continuity of $\nabla^2 f$ obviates the need for such
conditions, allowing the following simplified definition.

\begin{definition}\label{def:strictly-simple}
  Let \func have $\SmHess$-Lipschitz continuous Hessian. We call $f$
  \emph{$\StrConv$-strict-saddle} if it is
  $(\StrConv^2/\SmHess, \StrConv, \StrConv)$-strict-saddle.
\end{definition}

With this definition in mind, we present Algorithm~\ref{alg:ASS}, which leverages
Algorithm~\ref{alg:ACCNC} to obtain linear convergence (in the desired accuracy
$\epsilon$) to a local minimizer of strict-saddle functions. The algorithm
proceeds in two phases, first finding a region of strong convexity, and in the
second phase solving a regularized version of resulting locally convex problem
in this region.  That the first phase of Alg.~\ref{alg:ASS} terminates in a
neighborhood of a local optimum of $f$, where $f$ is convex in this
neighborhood, is an immediate consequence of the strict-saddle property coupled
with the gradient and Hessian bounds of Theorem~\ref{thm:complexity-of-ACCNC}. We
can then apply (accelerated) gradient descent to quickly find the local
optimum, which we describe rigorously in the following theorem.

\def\ASS{Accelerated-strict-saddle-method}
\newcommand{\callASS}[1]{\hyperref[alg:ASS]{\Call{\ASS}{#1}}}
\newcommand{\xloc}{x^{\star}_{+}}
\newcommand{\xf}{x}

\begin{algorithm}
	\caption{Acceleration of smooth strict-saddle optimization}\label{alg:ASS}
	\begin{algorithmic}[1]
		\Function{\ASS}{$x_1$, $f$, $\epsilon$, $\SmGrad$, $\SmHess$, 
		$\StrConv$, $\DeltaF$, $\delta$}
		\Statex \textbf{Phase one}
		\State Set $\varepsilon = \max\left\{\epsilon, 
		\frac{\StrConv^2}{16\SmHess}\right\}$
		\State Set  $\alpha = \min \left\{ \SmGrad, \max \left\{ \varepsilon^2 
		\DeltaF^{-1}, \varepsilon^{1/2}\SmHess^{1/2} \right\} \right\}$ 
		\Comment{as in Theorem~\ref{thm:complexity-of-ACCNC}}
		\State $x_{+}\gets \callACCNC{x_1, f, \varepsilon, \SmGrad, \SmHess, 
		\alpha, \DeltaF, \delta}$
              \label{line:x-plus-strict}
		\Statex \textbf{Phase two:}
		\If{$\epsilon < \varepsilon$} \Comment{non-trivial case}
		\State Set ${f_{+}}(x) = f(x) +  \SmGrad \hinge{\| x-x_+ \| - 
		\frac{\StrConv}{4 \SmHess}}^2$
		\State \Return{\callAGD{${f_{+}}$, $x_{+}$, $\epsilon$, $5\SmGrad$, 
		$\StrConv/2$}}
		\Else
		\State \Return{$x_{+}$}
		\EndIf
		\EndFunction
	\end{algorithmic}
\end{algorithm}

\begin{theorem}\label{theorem:ASS}
  Let \func be $\SmGrad$-smooth, have $\SmHess$-Lipschitz continuous Hessian,
  and be $\StrConv$-strict-saddle.  Let $\epsilon \ge 0$ and
  $\delta \in (0,1)$. With probability at least $1 - \delta$, \callASS{$x_1$,
    $f$, $\epsilon$, $\SmGrad$, $\SmHess$, $\StrConv$, $\DeltaF$, $\delta$}
  returns a point $\xf$ that satisfies $\|\nabla f(\xf)\| \le \epsilon$ in
  time
  \begin{flalign*}
    O\left( \T \left[ \sqrt{\frac{\SmGrad}{\StrConv}}\log\left( \tau' + 
          \frac{1}{\epsilon} \right) + 
	\frac{\SmGrad^{1/2}\SmHess^2\DeltaF}{\StrConv^{7/2}}\log\tau' \right] 
    \right),
  \end{flalign*}
  where $\tau' = 1 + \SmGrad/\StrConv + 1 / \delta  + \Dim + \SmHess + 
  \DeltaF$. When $\epsilon \le \frac{\StrConv^2}{16\SmHess}$, 
  with the same probability there exists a local minimizer $\xloc$ of $f$
  such that
  \begin{flalign}\label{eq:xloc-guarantee}
    \norm{\xf-\xloc} \le \frac{2\epsilon}{\StrConv}~~\mbox{and}~~
    f(\xf)-f(\xloc) \le \frac{2\SmGrad\epsilon^2}{\StrConv^2}.
	\end{flalign}
\end{theorem}

\begin{proof}
  The result in the low accuracy regime in which
  $\epsilon > \frac{\StrConv^2}{16\SmHess}$ is immediate by
  Theorem~\ref{thm:complexity-of-ACCNC}, and we therefore focus on the case
  that $\epsilon \le \frac{\StrConv^2}{16\SmHess}$. We perform our analysis
  conditional on the event, which holds with probability at least $1-\delta$,
  that the guarantees of Theorem~\ref{thm:complexity-of-ACCNC} hold. That is,
  that $x_+$ generated in Line~\ref{line:x-plus-strict} satisfies
  \begin{equation}
    \norm{\nabla f(x_+)} \le \frac{\StrConv^2}{16\SmHess}~~\mbox{and}~~
    \nabla^2 f(x_+) \succeq 
    -\frac{\StrConv}{4}I,\label{eq:xplus-guarantee}
  \end{equation}
  and that it is computed in time
  \begin{flalign*}
    T_1 & =
    O\left( \T \left[  \left( 
          \frac{\SmGrad^{1/2}\SmHess^2\DeltaF}{\StrConv^{7/2}} + 
          \frac{\SmGrad^{1/2}\SmHess\DeltaF^{1/2}}{\StrConv^2} + 
          1\right)\log\tau' \right] \right)\\
    &\stackrel{(i)}{=}
    O\left( \T \sqrt{\frac{\SmGrad}{\StrConv}}\left[ 
	\frac{\SmHess^2\DeltaF}{\StrConv^3} + 1 \right]\log\tau' \right),
  \end{flalign*}
  where $\tau'$ is as in the theorem statement. Equality $(i)$ is a
  consequence of the inequalities $1 \le \sqrt{\SmGrad/\StrConv}$ and
  $1+a+a^2 = O(1+a^2)$ for $a\ge0$.
  
  In conjunction with Definition~\ref{def:strictly-simple}, the
  bounds~\eqref{eq:xplus-guarantee} imply that
  $\nabla^2 f(x_+) \succeq \StrConv I$.
  Recalling Lemma~\ref{lem:f-negatively-convex},
  and the bound~\eqref{eq:rhohess} from its proof,
  a trivial calculation involving the Lipschitz continuity of
  $\nabla^2 f$ shows that
  $f_+(x) = f(x) + \SmGrad \hinge{\norm{x - x_+} - \StrConv / 4 \SmHess}^2$
  is $\StrConv / 2$-strongly convex. Additionally, we have
  immediately that $f_+$ is $5 \SmGrad$-smooth.

  Let $\xloc$ be the unique global minimizer of $f_+$. By the strong convexity
  of $f_+$, we may bound the distance between $x_+$ and $\xloc$ (recall
  Lemma~\ref{lem:grad:bounds}) by
  \begin{flalign*}
    \norm{x_+ - \xloc} \le \frac{2 \norm{\nabla f_+(x_+)}}{\StrConv} = 
    \frac{2 \norm{\nabla f(x_+)}}{\StrConv} \le 
    \frac{\StrConv}{8\SmHess},
  \end{flalign*}
  where final inequality is immediate from the gradient
  bound~\eqref{eq:xplus-guarantee}. By construction, $f_+ = f$ on the ball
  $\{x : \norm{x - x_+} \le \StrConv / 4 \SmHess\}$, and as $\xloc$ belongs to
  the interior of this ball, it is a local minimizer of $f$.
  
  Let $\xf$ be the point produced by the call to \callAGD{}. By 
  Lemma~\ref{lem:AGD}, $\xf$ satisfies $\norm{\nabla f_+(\xf)} \le 
  \epsilon$ and is computed in time
  \begin{equation*}
    T_2 \defeq \T + \T\sqrt{\frac{10\SmGrad}{\StrConv}}
    \log \left( \frac{200 \SmGrad^2 \DeltaF }{ \StrConv \epsilon^2
      } \right)  = O\left(\T \sqrt{\frac{\SmGrad}{\StrConv}}\log\left(\tau' + 
	\frac{1}{\epsilon}\right)\right).
  \end{equation*}
  The strong convexity of $f_+$ once more (Lemma~\ref{lem:grad:bounds}) implies
  \begin{equation*}
    \norm{\xf - \xloc} \le \frac{\norm{\nabla f_+(\xf)}}{\StrConv/2} \le 
    \frac{2\epsilon}{\StrConv} \le  
    \frac{\StrConv}{8\SmHess},
  \end{equation*}
  which gives the distance bound in
  expression~\eqref{eq:xloc-guarantee}. Combining
  $\norm{\xf - \xloc} \le \frac{\StrConv}{8\SmHess}$ and
  $\norm{x_+ - \xloc} \le \frac{\StrConv}{8\SmHess}$, we have that
  $\norm{\xf - x_+} \le \frac{\StrConv}{4\SmHess}$, and therefore
  $f(\xf) = f_+(\xf)$ and
  $\norm{\nabla f(\xf)} = \norm{\nabla f_+(\xf)} \le \epsilon$. The
  functional
  bound~\eqref{eq:xloc-guarantee} then follows from the $\SmGrad$-smoothness of
  $f$ and that $\nabla f(\xloc) = 0$,
  as
  \begin{equation*}
    f(\xf) - f(\xloc) \le
    \nabla f(\xloc)^T(\xf - \xloc)
    + \frac{\SmGrad}{2}\norm{\xf - \xloc}^2
    = \frac{\SmGrad}{2}\norm{\xf - \xloc}^2 .
  \end{equation*}
  The
  running time guarantee follows by summing $T_1$ and $T_2$ above.
\end{proof}

\section*{Acknowledgment}
OH was supported by the PACCAR INC 
fellowship. YC and JCD were partially supported by
the SAIL-Toyota Center for AI Research. YC was partially supported 
by the Stanford Graduate Fellowship and the Numerical Technologies 
Fellowship. JCD
was partially supported by
the National Science Foundation award NSF-CAREER-1553086

\newpage

\bibliographystyle{abbrvnat}
\bibliography{bio}

\end{document}